\documentclass[12pt]{article}
\usepackage{amsmath,latexsym,amssymb}
\usepackage{enumerate}
\usepackage{amsthm}

\newcommand\NN{\mathbb{N}}
\newcommand\RR{\mathbb{R}}

\DeclareMathOperator\sP{P}   
\newcommand{\rP}{\mathrm{P}} 

\newtheorem{theorem}{Theorem}[section]
\newtheorem{definition}[theorem]{Definition}

\newtheorem{lemma}[theorem]{Lemma}
\newtheorem{remark}[theorem]{Remark}

\newcommand{\address}{Address: Department of Mathematics, University of North Texas, 1155 Union Circle \#311430, Denton, TX 76203-5017, USA; E-mail: allaart@unt.edu, JoseIslas@my.unt.edu}

\title{A sharp lower bound for choosing the maximum of an independent sequence}
\author{Pieter C. Allaart and Jos\'e A. Islas\footnote{\address}}

\begin{document}

\maketitle

\begin{abstract}
This paper considers a variation of the full-information secretary problem where the random variables to be observed are independent but not necessary identically distributed. The main result is a sharp lower bound for the optimal win probability. Precisely, if $X_1,\dots,X_n$ are independent random variables with known continuous distributions and $V_n(X_1,\dots,X_n):=\sup_\tau \sP(X_\tau=M_n)$, where $M_n:=\max\{X_1,\dots,X_n\}$ and the supremum is over all stopping times adapted to $X_1,\dots,X_n$, then
$$V_n(X_1,\dots,X_n)\geq \left(1-\frac{1}{n}\right)^{n-1},$$
and this bound is attained. The method of proof consists in reducing the problem to that of a sequence of random variables taking at most two possible values, and then applying Bruss' sum-the-odds theorem (2000). In order to obtain a sharp bound for each $n$, we improve Bruss' lower bound (2003) for the sum-the-odds problem.


\end{abstract}

\section{Introduction}

In the classical secretary problem or best-choice problem, a known number $n$ of applicants for a single position are interviewed one by one in a random order, and after each interview a manager must decide whether or not to hire the applicant just interviewed, based solely on his or her relative rank among the applicants interviewed thus far. The objective is to hire the best applicant. It is well known that under the optimal strategy, the probability of success decreases monotonically to $1/e\approx .3679$ as $n\to\infty$. An entertaining account of the history of the secretary problem and some of its many variations can be found in Ferguson \cite{Ferguson}; see also Samuels \cite{Samuels}.

One natural extension is the full-information best-choice problem, in which each applicant, independent of the others, can be assigned a numerical score whose distribution is known in advance, and the objective is to hire the applicant with the highest score. The mathematical framework for this problem is as follows. Here and throughout this paper, let $X_1,\dots,X_n$ be independent random variables with known continuous distributions, let $M_n:=\max\{X_1,\dots,X_n\}$, and define
\begin{equation}
V_n(X_1,\dots,X_n):=\sup_\tau \sP(X_\tau=M_n),
\label{eq:optimal-stopping-problem}
\end{equation}
where the supremum is over all stopping times adapted to the natural filtration of $X_1,\dots,X_n$. Gilbert and Mosteller \cite{GilMost} solved this optimal stopping problem in the case when $X_1,\dots,X_n$ are independent and identically distributed (i.i.d.). They showed that the optimal win probability is independent of the distribution of the $X_i$'s, and decreases monotonically to $.5802$.

It is natural to ask how much lower the optimal win probability can be if we drop the assumption that the random variables are identically distributed. In this case, it should be fairly clear that the optimal win probability depends on the distributions of $X_1,\dots,X_n$, and we aim to find a sharp lower bound. Our main result is

\begin{theorem} \label{thm:main}
For each $n\in\NN$, we have
\begin{equation}
V_n(X_1,\dots,X_n)\geq \left(1-\frac{1}{n}\right)^{n-1},
\label{eq:main-lower-bound}
\end{equation}
and this bound is attained.
\end{theorem}

\begin{remark}
{\rm
The bound in the theorem decreases to $1/e\approx .3679$ as $n\to\infty$. This is quite a bit smaller than the limiting win probability $.5802$ in the i.i.d. case, which should not come as a surprise. (Intuitively, when the distributions become progressively more spread out -- as they do in the extremal case; see Lemma \ref{lem:V-sequence-bound} below -- there is more uncertainty about the future than when all distributions are identical, resulting in a lower optimal win probability.) Comparing our bound with the classical secretary problem, we see moreover that, in the worst case and for large $n$, a gambler who has full information can do no better than a gambler who observes only relative ranks.
}
\end{remark}

\begin{remark}
{\rm
One might ask whether the continuity assumption about the $X_i$'s is needed in Theorem \ref{thm:main}. The way our objective function is defined in \eqref{eq:optimal-stopping-problem} implies that we win if we stop with a value that is at least tied for the overall maximum. Thus, it seems that the possibility of ties should only make it easier to win and consequently, the lower bound \eqref{eq:main-lower-bound} should continue to hold when the random variables $X_1,\dots,X_n$ are permitted to have atoms. However, it is not clear to us how to prove this formally.
}
\end{remark}

Our method of proof consists of two steps. First, we construct a sequence $Y_1,\dots,Y_n$ of simple random variables (each of which, in fact, takes on at most two possible values) such that $V_n(X_1,\dots,X_n)\geq V_n(Y_1,\dots,Y_n)$. Next, we use Bruss' ``sum-the-odds" theorem \cite{Bruss1} to obtain an explicit expression for $V_n(Y_1,\dots,Y_n)$, which we then minimize by establishing a lower bound for the ``sum-the-odds" problem which improves on that given by Bruss \cite{Bruss2}. This lower bound, though not difficult to obtain, is interesting in its own right.

We have not found a more direct way to prove the main result, though we do not rule out the possibility that one exists. However, we believe that the technique of proof used here is of independent interest and may be applicable to other optimal stopping problems of a similar nature.

\section{Bruss' ``sum-the-odds" theorem}

Bruss \cite{Bruss1} considered the problem of stopping at the last success in a sequence of independent Bernoulli trials. Specifically, let $A_1,\dots,A_n$ be independent events with $\rP(A_i)=p_i$, $i=1,\dots,n$, and let $I_i:=I_{A_i}$, the indicator random variable corresponding to $A_i$, for $i=1,\dots,n$. If we think of the value 1 as representing a success and 0 as representing a failure, the problem of stopping at the last success comes down to finding a stopping time $\tau$ that maximizes $\rP(I_\tau=1,I_{\tau+1}=\dots=I_n=0)$. 

\begin{theorem}[Sum-the-odds theorem, \cite{Bruss1}] \label{thm:Bruss-odds}
Let $I_1, I_2,\dots,I_n$ be a sequence of independent indicator random variables with $p_j=E(I_j)$. Let $q_j=1-p_j$ and $r_j=p_j/q_j$. 
Consider the problem of stopping at the last success; that is, the optimal stopping problem
\begin{equation*}
v:=v(p_1,\dots,p_n):=\sup_\tau \sP(I_\tau=1,I_{\tau+1}=\dots=I_n=0).
\end{equation*}
Then the optimal rule is to stop on the first index (if any) $k$ with $I_k=1$ and $k\geq s$, where 
\begin{equation*}
s:=\sup \left\{1,\sup\left\{1\leq k \leq n: \sum_{j=k}^n r_j \geq 1 \right\} \right\},
\end{equation*}
with $\sup\{\emptyset\}:=-\infty$. Moreover, the optimal win probability is given by
\begin{equation*}
v=v(p_1,\dots,p_{n})=\left(\prod_{j=s}^n q_j\right) \left(\sum_{k=s}^n r_k\right).
\end{equation*}
\end{theorem}

To see how this may be applied to the best-choice problem, let $A_i$ be the event that the $i$th applicant is best so far. In the classical secretary problem, where only relative ranks are observed and all $n!$ orderings of the applicants are equally likely, the events $A_1,\dots,A_n$ are independent, so Bruss' theorem applies to give the optimal stopping rule. However, in the full-information case we are considering here, we have $A_i=\{X_i\geq \max\{X_1,\dots,X_{i-1}\}\}$ for $i=1,\dots,n$, and these events are generally {\em not} independent when $X_1,\dots,X_n$ are not identically distributed. Consider, for instance, the case when $n=3$, $X_1\equiv 1$, $X_2=0$ or $3$ each with probability $1/2$, and $X_3=0,2$ or $4$ each with probability $1/3$; one checks easily that $A_2$ and $A_3$ are not independent.

To overcome this issue, we will reduce the problem to that of a sequence of random variables which is simple enough that the sum-the-odds theorem {\em does} apply to it; see the next section. But first we need a sharp lower bound for the optimal win probability $v(p_1,\dots,p_n)$. 

Bruss \cite{Bruss2} showed that if $\sum_{k=1}^n r_k\geq 1$, then $v(p_1,\dots,p_n)>e^{-1}$, and this is the best {\em uniform} lower bound. However, for fixed $n$ we can do somewhat better. Let $b_n$ denote the lower bound in Theorem \ref{thm:main}, that is,
\begin{equation*}
b_n:=\left(1-\frac{1}{n}\right)^{n-1}.
\end{equation*}

\begin{theorem} \label{thm:odds-bound}
If $\sum_{k=1}^n r_k\geq 1$, then 
\begin{equation*}
v=v(p_1,\dots,p_n)\geq b_{n+1}=\left(1-\frac{1}{n+1}\right)^n.
\end{equation*}
The bound is attained when $p_1=p_2=\dots=p_n=1/(n+1)$.
\end{theorem}

\begin{lemma} \label{lem:decreasing-bounds}
The sequence $(b_n)$ is strictly decreasing, and $\lim_{n\to\infty}b_n=e^{-1}$.
\end{lemma}

\begin{proof}
This follows since $b_{n+1}^{-1}=\left(1+\frac{1}{n}\right)^n$, which increases in $n$ and has limit $e$.
\end{proof}

\begin{proof}[Proof of Theorem \ref{thm:odds-bound}]
We initially follow Bruss' proof \cite{Bruss2} of the bound $e^{-1}$.
For convenience, we reindex the $p_k$'s, and hence the $q_k$'s and $r_k$'s, by redefining $p_i:=p_{n-i+1}$, $q_i:=q_{n-i+1}$ and $r_i:=r_{n-i+1}$, for $i=1,2,\dots,n$. In this new notation, let $R_k:=r_1+\dots+r_k$ and let $t:=\inf\{k:R_k\geq 1\}$. Then we can write $v$ as
\begin{equation}
v=R_t \prod_{k=1}^t q_k.
\label{eq:v-formula}
\end{equation}

If $t=1$, then $r_1\geq 1$, and hence $v=r_1 q_1=p_1\geq 1/2$. Thus, the statement is true for $t=1$. Assume therefore $t\geq 2$. Use $q_j=(1+r_j)^{-1}$ to rewrite $v$ in the form
\begin{equation} \label{eq:v-expanded}
v=\frac{R_t}{(1+r_t)\prod_{j=1}^{t-1}(1+r_j)}.
\end{equation}
Using the geometric mean-arithmetic mean inequality we have
\begin{equation*}
\left(\prod_{j=1}^{t-1}(1+r_j) \right)^{\frac{1}{t-1}} \leq \frac{\sum_{j=1}^{t-1} (1+r_j)}{t-1}=1+\frac{R_{t-1}}{t-1},
\end{equation*}
and substituting this into \eqref{eq:v-expanded} gives
\begin{equation*}
v\geq \frac{R_t}{(1+r_t)\left(1+\frac{R_{t-1}}{t-1}\right)^{t-1}}.
\end{equation*}
We now refine the analysis from the second part of Bruss' proof. It is simpler to consider $v^{-1}$:
\begin{equation*}
v^{-1}\leq \frac{(1+r_t)\left(1+\frac{R_{t-1}}{t-1}\right)^{t-1}}{R_{t-1}+r_t}=g_t(R_{t-1},r_t),
\end{equation*}
where
\begin{equation*}
g_t(x,y):=\frac{1+y}{x+y}\left(1+\frac{x}{t-1}\right)^{t-1}.
\end{equation*}
In view of
\begin{equation*}
R_{t-1}=\sum_{k=1}^{t-1}r_k<1\leq R_t=R_{t-1}+r_t,
\end{equation*}
we need to maximize $g_t(x,y)$ over $0\leq x <1$, $1-x\leq y$. For fixed $x$ in this range, $g_t(x,y)$ is clearly decreasing in $y$, so that
\begin{equation*}
g_t(x,y)\leq g_t(x,1-x)=(2-x)\left(1+\frac{x}{t-1}\right)^{t-1}.
\end{equation*}
Elementary calculus shows that this function is maximized at $x=x^*:=(t-1)/t$, and so
\begin{equation*}
v^{-1}\leq g_t(x^*,1-x^*)=\left(\frac{t+1}{t}\right)^t.
\end{equation*}
Hence,
\begin{equation*}
v \geq \left(\frac{t}{t+1}\right)^t=b_{t+1}\geq b_{n+1},
\end{equation*}
where the last inequality follows from Lemma \ref{lem:decreasing-bounds}. When $p_1=p_2=\dots=p_n=1/(n+1)$, we have $q_k=n/(n+1)$ and $r_k=1/n$ for each $k$, so that $t=n$ and $R_t=1$, and \eqref{eq:v-formula} yields $v=b_{n+1}$.
\end{proof}

\section{Proof of the main theorem}

We first introduce some notation and terminology.
For $j=1,\dots,n$, let $F_j$ denote the distribution function of $X_j$. By our assumption, each $F_j$ is continuous.
For $k=1,\dots,n$, we shall call $X_k$ a {\em candidate} if $X_k=\max\{X_1,\dots,X_k\}$. If we observe a candidate $X_k=x$ and we stop, we win if and only if none of the future observations $X_{k+1},\dots,X_n$ exceeds $x$. This happens with probability
\begin{equation}
U_k(x):=\sP(X_{k+1}\leq x,\dots,X_n\leq x)=\prod_{j=k+1}^n F_j(x).
\label{eq:U-def}
\end{equation}
On the other hand, if we have observed $X_1=x_1,\dots,X_k=x_k$ and we continue, we win if and only if we stop at some time $k<\tau\leq n$ and $X_\tau=\max\{m,X_{k+1},\dots,X_n\}$, where $m:=\max\{x_1,\dots,x_k\}$. Thus, the optimal win probability if we continue is
\begin{align*}
W_k(m):&=\sup_{k<\tau\leq n} \sP(X_\tau=M_n|\max\{X_1,\dots,X_k\}=m)\\
&=\sup_{k<\tau \leq n} \sP\left(X_{\tau}\geq \max\{m,X_{k+1},\dots,X_n\} \right).
\end{align*}
It is straightforward to verify, by conditioning on the value of $X_{k+1}$, the recursion
\begin{equation}
W_k(m)=F_{k+1}(m) W_{k+1}(m)+\int_m^{\infty} \max\{U_{k+1}(x),W_{k+1}(x)\}dF_{k+1}(x),
\label{eq:W-recursion}
\end{equation}
for $k=1,\dots,n-1$, since an observation smaller than $m$ at time $k+1$ would force us to continue.

The following definition identifies a class of sequences of random variables which will turn out to be extremal for our problem.

\begin{definition} \label{def:V-sequence}
A sequence $Y_1,\dots,Y_n$ of independent random variables will be called a {\em V-sequence} if there exist constants $a_1,a_2,\dots,a_n$ and $b_2,\dots,b_n$ with $a_1<a_2<\dots<a_n$ and $a_1>b_2>b_3>\dots>b_n$, and
\begin{enumerate}[(i)]
\item $\rP(Y_1=a_1)=1$, and
\item $\rP(Y_i=a_i\ \mbox{or}\ b_i)=1$ for $i=2,\dots,n$.
\end{enumerate}
\end{definition}

We will show first that, if $Y_1,\dots,Y_n$ is any $V$-sequence, then the optimal win probability $V_n(Y_1,\dots,Y_n)$ for this sequence is greater than or equal to the bound in Theorem \ref{thm:main}. Then, for a given sequence $X_1,\dots,X_n$ of independent continuous random variables, we will construct a specific $V$-sequence $Y_1,\dots,Y_n$ (whose distributions depend on those of $X_1,\dots,X_n$, of course) such that $V_n(X_1,\dots,X_n)\geq V_n(Y_1,\dots,Y_n)$. This will clearly establish the lower bound of Theorem \ref{thm:main}.

\begin{lemma} \label{lem:V-sequence-bound}
If $Y_1,\dots,Y_n$ is a V-sequence, then
\begin{equation*}
V_{n}(Y_1,\dots,Y_n)\geq \left(1-\frac{1}{n}\right)^{n-1},
\end{equation*}
and this bound is attained.
\end{lemma}

\begin{proof}
For $k=1,\dots,n-1$, we define the indicator random variables
\begin{equation*}
I_k:=\begin{cases}
1 & \mbox{if }Y_{k+1} > \max\{Y_1,\dots,Y_k\}, \\
0 & \mbox{otherwise}.
\end{cases}
\end{equation*}
Let $p_k:=E(I_k)=P(Y_{k+1}=a_{k+1})$, $q_k:=1-p_k$ and $r_k:=p_k/q_k$. We have that 
\begin{equation}
V_n(Y_1,\dots,Y_n)=\max\{U_1(a_1),W_1(a_1)\}. 
\label{eq:stop-or-continue}
\end{equation}
Notice that if we decide to continue after observing $Y_1$, we win if and only if we choose the last success in the sequence $I_1,\dots,I_{n-1}$, which is an independent sequence since for each $k\geq 1$, $I_k=1$ if and only if $Y_{k+1}=a_{k+1}$. Thus, $W_1(a_1)=v(p_{1},\dots,p_{n-1})$, so if $\sum_{k=1}^{n-1} r_k \geq 1$, Theorem \ref{thm:odds-bound} gives
\begin{equation*}
W_1(a_1)\geq \left(1-\frac{1}{n}\right)^{n-1}.
\end{equation*}
Suppose now that $\sum_{k=1}^{n-1} r_k \leq 1$. Stopping at the first observation gives us win probability
\begin{equation*}
U_1(a_1)=q_1\cdots q_{n-1}=\prod_{i=1}^{n-1}(1+r_i)^{-1}.
\end{equation*}
Using the arithmetic mean-geometric mean inequality we have
\begin{align*}
\left(\prod_{i=1}^{n-1}(1+r_i) \right)^{\frac{1}{n-1}} \leq \frac{ \sum_{k=1}^{n-1} (1+r_k )}{n-1} \leq \frac{n}{n-1}.
\end{align*}
Thus,
\begin{equation*}
U_1(a_1)\geq \left(\frac{n}{n-1}\right)^{-(n-1)}=\left(1-\frac{1}{n}\right)^{n-1}.
\end{equation*}
In both cases, \eqref{eq:stop-or-continue} yields the desired lower bound.

To see that the bound is attained, let $p_i:=\sP(Y_{i+1}=a_{i+1})=1/n$, $q_i:=1-p_i$ and $r_i:=p_i/q_i=1/(n-1)$, for $i=1,\dots,n-1$. If we skip the first observation and play optimally from then on, we win if and only if we stop at the last success of $I_1,\dots,I_{n-1}$, that is, $W_1(a_1)=v(p_{1},\dots,p_{n-1})$. Note that $\inf\{s:\sum_{j=s}^n r_j\geq 1\}=2$. Theorem \ref{thm:Bruss-odds} gives  
\begin{equation*}
v(p_{1},\dots,p_{n-1})=\sum_{j=2}^n \frac{1}{n} \left(\frac{n-1}{n}\right)^{n-2}= \left(1-\frac{1}{n}\right)^{n-1}.
\end{equation*}
If, on the other hand, we choose to stop at the first observation, the win probability is $U_1(a_1)=q_1\cdots q_{n-1}=\left(1-\frac{1}{n}\right)^{n-1}$ as well. Hence
\begin{equation*}
V_n(Y_1,\dots,Y_n)=\left(1-\frac{1}{n}\right)^{n-1},
\end{equation*}
as we wanted to show.
\end{proof}

\begin{lemma} \label{lem:critical-values}
For the optimal stopping problem \eqref{eq:optimal-stopping-problem}, there exists a sequence of critical values $x_1^*,\dots,x_{n-1}^*$ such that at observation $1\leq k <n$ it is optimal to stop if and only $X_k\geq x_k^*$ and $X_k$ is a candidate. In other words, the optimal stopping rule is
\begin{equation*}
\tau^*:=\min\big\{k\leq n: X_k\geq\max\{X_1,\dots,X_{k-1},x_k^*\}\big\},
\end{equation*}
or $\tau^*=n$ if no such $k$ exists.
\end{lemma}

\begin{proof}
For each $i=1,\dots,n-1$, $U_{i}(x)$ is continuous and increasing, since it is the product of continuous and increasing functions. Moreover, $\lim_{x\to-\infty}U_i(x)=0$ and $\lim_{x\to\infty}U_i(x)=1$. For each $i=1,\dots,n-1$, we have 
\begin{equation*}
W_{i}(x)=\sup_{i<\tau \leq n }P\left(X_{\tau}\geq \max\{x,X_{i+1},\dots,X_n\} \right).
\end{equation*}
This shows that $W_i$ is continuous and decreasing. We claim that
\begin{equation}
\lim_{x\to-\infty}W_i(x)>0
\label{eq:W-at-left}
\end{equation}
and
\begin{equation}
\lim_{x\to\infty}W_i(x)<1.
\label{eq:W-at-right}
\end{equation}
It then follows that the graphs of $U_i$ and $W_i$ intersect each other at some (not necessarily unique) point $x_i^*$. 

To see \eqref{eq:W-at-left}, choose $k\in\{i+1,\dots,n\}$ such that $P(X_k=\max\{X_{i+1},\dots,X_n\})>0$. Then
\begin{align*}
L:&=\lim_{x\to-\infty} P(X_k\geq x, X_k=\max\{X_{i+1},\dots,X_n\})\\
&= P(X_k=\max\{X_{i+1},\dots,X_n\})\\
&>0,
\end{align*}
so, by considering the stopping rule $\tau\equiv k$,
\begin{align*}
W_i(x)
&\geq P(X_k\geq\max\{x,X_{i+1},\dots,X_n\})\\
&=P(X_k\geq x, X_k=\max\{X_{i+1},\dots,X_n\})\\
&\to L>0, \qquad\mbox{as $x\to-\infty$}.
\end{align*}
This gives \eqref{eq:W-at-left}. Next, choose $x$ such that $P(\max\{X_{i+1},\dots,X_n\}\geq x)<1$. Then
\begin{equation*}
W_i(x)\leq \sup_{i<\tau\leq n} P(X_\tau\geq x)\leq P(\max\{X_{i+1},\dots,X_n\}\geq x)<1,
\end{equation*}
which implies \eqref{eq:W-at-right}.
\end{proof}

When $X_1,\dots,X_n$ are i.i.d., the critical values $x_1^*,\dots,x_{n-1}^*$ form a decreasing sequence; see \cite{GilMost}. This remains the case if $X_1,\dots,X_n$ are merely assumed to be independent.

\begin{lemma} \label{lem:decreasing-critical-values}
We have $x_1^*\geq x_2^*\geq \dots \geq x_{n-1}^*$.
\end{lemma}

\begin{proof}
It suffices to show that
\begin{equation*}
W_{k+1}(x)\geq U_{k+1}(x) \quad \Longrightarrow \quad W_k(x)\geq U_k(x), \qquad x\in\RR, \quad k=1,\dots,n-2.
\end{equation*}
But this is almost obvious, since assuming the inequality on the left, \eqref{eq:W-recursion} gives 
\begin{equation*}
W_k(x)\geq F_{k+1}(x)W_{k+1}(x)\geq F_{k+1}(x)U_{k+1}(x)=U_k(x),
\end{equation*}
where the last equality is a consequence of \eqref{eq:U-def}.
\end{proof}

\begin{lemma} \label{lem:using-IVT}
For each $i\in\{2,3,\dots,n-1\}$, there exists a constant $c_i>x_1^*$ such that
\begin{equation}
\int_{x_1^*}^{\infty} U_{i}(x)dF_{i}(x) = U_{i}(c_{i})P(X_{i}>x_1^*).
\label{eq:intermediate-value}
\end{equation}
\end{lemma}

\begin{proof}
If $\rP(X_i>x_1^*)=0$, we can choose $c_i>x_1^*$  arbitrarily, and both sides of \eqref{eq:intermediate-value} will be zero. So assume $\rP(X_i>x_1^*)>0$. Since $U_i$ is increasing and bounded by $1$, we have
\begin{align*}
U_{i}(x_1^*)P(X_{i}>x_1^*)\leq \int_{x_1^*}^{\infty} U_{i}(x)dF_{i}(x) \leq P(X_{i}>x_1^*),
\end{align*}
so
\begin{equation*}
U_{i}(x_1^*) \leq \frac{1}{P(X_{i}>x_1^*)}\int_{x_1^*}^{\infty} U_{i}(x)dF_{i}(x) \leq 1.
\end{equation*}
If the inequality on the left is strict, then, since $U_i(x)$ is continuous and tends to $1$ as $x \rightarrow \infty$, we can apply the intermediate value theorem to obtain $c_{i}>x_1^*$ such that
\begin{equation*}
U_{i}(c_{i})=\frac{1}{P(X_{i}>x_1^*)}\int_{x_1^*}^{\infty} U_{i}(x)dF_{i}(x).
\end{equation*}
Otherwise, it must be the case that $U_i$ is constant (and hence equal to $1$) on $[x_1^*,\infty)$, and we can choose $c_i$ to be any number strictly greater than $x_1^*$. In either case, we get \eqref{eq:intermediate-value}.
\end{proof}

\begin{lemma}[Reduction lemma] \label{lem:reduction}
There exists a V-sequence $X_1^{'},\dots,X_n^{'}$ such that $V_n(X_1,\dots,X_n)\geq V_n(X_1^{'},\dots,X_n^{'})$.
\end{lemma}

\begin{proof}
Let $c_2,\dots,c_{n-1}$ be the constants from Lemma \ref{lem:using-IVT}.
Choose numbers $a_2\dots,a_n$ and $b_2,\dots,b_{n}$ such that $x_1^*<a_2<a_3<\dots<a_n$ and $a_j\leq c_j$ for $j=2,\dots,n-1$, and $x_1^*>b_2>\dots>b_{n}$. Set $X_1^{'}:\equiv x_1^*$ and for $2\leq i \leq n$, define
\begin{equation*}
X_i^{'}:=\begin{cases}
a_i, & \mbox{if}\ X_i > x_1^*\\
b_i, & \mbox{if}\ X_i \leq x_1^*.
\end{cases}
\end{equation*}
Note that the random variables $X_1',\dots,X_n'$ are independent. The idea of the proof is to replace the random variables $X_1,\dots,X_n$ by their counterparts $X_1',\dots,X_n'$ one at a time, starting with $X_1$, and to show that each such replacement does not increase the optimal win probability. In order to do so, we need to introduce analogs of the functions $U_j$ and $W_j$ for the sequence $X_1^{'},\dots,X_k^{'},X_{k+1},\dots,X_n$, where $k=0,1,\dots,n$. First, introduce random variables
\begin{equation*}
Y_j^{(k)}:=\begin{cases}
X_j' & \mbox{if $1\leq j\leq k$},\\
X_j  & \mbox{if $k<j\leq n$},
\end{cases}
\qquad k=0,1,\dots,n, \quad j=1,2,\dots,n.
\end{equation*}
Now define
\begin{gather*}
U_j^{(k)}(x):=\sP(Y_{j+1}^{(k)}\leq x,\dots,Y_n^{(k)}\leq x)=\sP(Y_{j+1}^{(k)}\leq x)\cdots\sP(Y_n^{(k)}\leq x),\\
W_j^{(k)}(x):=\sup_{j<\tau\leq n} \sP\left(Y_\tau^{(k)}=\max\{x,Y_{j+1}^{(k)},\dots,Y_n^{(k)}\}\right),
\end{gather*}
for $k=0,1,\dots,n$ and $j=1,2,\dots,n$. We also write
\begin{equation*}
V_n^{(k)}:=V_n\left(Y_1^{(k)},\dots,Y_n^{(k)}\right), \qquad k=0,1,\dots,n,
\end{equation*}
so $V_n^{(0)}=V_n(X_1,\dots,X_n)$. Observe that
\begin{equation}
k\leq j \quad \Longrightarrow \quad W_j^{(k)}=W_j\ \mbox{and}\ U_j^{(k)}=U_j.
\label{eq:drop-the-superscript}
\end{equation}
And since 
\begin{equation*}
\rP(X_i'\leq x_1^*)=\sP(X_i'=b_i)=\sP(X_i\leq x_1^*), \qquad i=2,3,\dots,n,
\end{equation*}
we also have
\begin{equation}
U_j^{(k)}(x_1^*)=U_j(x_1^*), \qquad j=1,\dots,n-1, \quad k=0,1,\dots,n.
\label{eq:U-not-changing}
\end{equation}
We will show that for all $0\leq k \leq n-1$, $V_n^{(k)}\geq V_n^{(k+1)}$. First,
\begin{align*}
V_n^{(0)}&=V_n(X_1,X_2,\dots,X_n)=E[\max\{U_1(X_1),W_1(X_1)\}] \\
&=\int_{-\infty}^{\infty} \max\{U_1(x),W_1(x)\}dF_1(x) \\
&= \int_{-\infty}^{x_1^*}  W_1(x) dF_1(x)+ \int_{x_1^*}^{\infty} U_1(x) dF_1(x) \\ 
&\geq \int_{-\infty}^{x_1^*} W_1(x_1^*) dF_1(x)+ \int_{x_1^*}^{\infty} U_1(x_1^*) dF_1(x) \\ 
&=U_1(x_1^*)=W_1(x_1^*) \\
&=V_n(X_1^{'},X_2,\dots,X_n)=V_n^{(1)},
\end{align*}
where the inequality follows since $W_1$ is decreasing and $U_1$ is increasing.

Next, we show for every $1 \leq k \leq n-1$ that $V^{(k)}_n\geq  V^{(k+1)}_n$. In view of \eqref{eq:U-not-changing}, this will follow if we prove that
\begin{equation}
W_j^{(k)}(x_1^*)\geq W_j^{(k+1)}(x_1^*), \qquad j\leq k<n,
\label{eq:comparison-array}
\end{equation}
for $j=1,\dots,n-1$; we can then take $j=1$.
We prove \eqref{eq:comparison-array} by backward induction on $j$. First, \eqref{eq:comparison-array} holds for $j=n-1$ since
\begin{equation*}
W_{n-1}^{(n-1)}(x_1^*)=\sP(X_n>x_1^*)=\sP(X_n'>x_1^*)=W_{n-1}^{(n)}(x_1^*).
\end{equation*}
Suppose \eqref{eq:comparison-array} holds with $j$ replaced by $j+1$; that is,
\begin{equation}
W_{j+1}^{(k)}(x_1^*)\geq W_{j+1}^{(k+1)}(x_1^*), \qquad j<k<n.
\label{eq:induction-hypothesis}
\end{equation}
We first claim that
\begin{equation}
U_{j+1}^{(k)}(a_{j+1})\geq W_{j+1}^{(k)}(a_{j+1}), \qquad j<k<n.
\label{eq:stop-if-top-value}
\end{equation}
This follows since $a_{j+1}>x_1^*\geq x_{k+1}^*$, $U_{j+1}^{(k)}$ is increasing and $W_{j+1}^{(k)}$ is decreasing, and so, for $j<k<n$,
\begin{align*}
U_{j+1}^{(k)}(a_{j+1})&\geq U_{j+1}^{(k)}(x_1^*)=U_{j+1}(x_1^*)\geq W_{j+1}(x_1^*)\\
&=W_{j+1}^{(j+1)}(x_1^*)\geq W_{j+1}^{(k)}(x_1^*)\geq W_{j+1}^{(k)}(a_{j+1}),
\end{align*}
where we used \eqref{eq:U-not-changing} in the first equality, \eqref{eq:drop-the-superscript} in the second equality, and the induction hypothesis \eqref{eq:induction-hypothesis} in the third inequality.

We will also need that
\begin{equation}
U_{j+1}^{(k)}(a_{j+1})\geq U_{j+1}^{(k+1)}(a_{j+1}), \qquad j<k<n.
\label{eq:U-comparison}
\end{equation}
To see this, note that
\begin{equation*}
\frac{U_{j+1}^{(k)}(a_{j+1})}{U_{j+1}^{(k+1)}(a_{j+1})}=\frac{\rP(X_{k+1}\leq a_{j+1})}{\rP(X_{k+1}'\leq a_{j+1})},
\end{equation*}
and
\begin{equation*}
\rP(X_{k+1}'\leq a_{j+1})=\sP(X_{k+1}'=b_{k+1})=\sP(X_{k+1}\leq x_1^*)\leq \sP(X_{k+1}\leq a_{j+1}),
\end{equation*}
since $k>j$ implies $a_{k+1}>a_{j+1}$. Now using \eqref{eq:stop-if-top-value}, the induction hypothesis \eqref{eq:induction-hypothesis}, \eqref{eq:U-comparison}, and finally \eqref{eq:stop-if-top-value} again, we obtain for $k>j$,
\begin{align*}
W_j^{(k)}(x_1^*)&=\sP(X_{j+1}'=b_{j+1})W_{j+1}^{(k)}(x_1^*)+\sP(X_{j+1}'=a_{j+1})U_{j+1}^{(k)}(a_{j+1})\\
&\geq \sP(X_{j+1}'=b_{j+1})W_{j+1}^{(k+1)}(x_1^*)+\sP(X_{j+1}'=a_{j+1})U_{j+1}^{(k+1)}(a_{j+1})\\
&=W_j^{(k+1)}(x_1^*).
\end{align*}
It remains to verify \eqref{eq:comparison-array} for $k=j$. Here we use \eqref{eq:intermediate-value} and the fact (by our choice of $a_{j+1}$) that $a_{j+1}\leq c_{j+1}$ to obtain
\begin{align*}
W_j^{(j)}(x_1^*)&=W_j(x_1^*)=\sP(X_{j+1}\leq x_1^*)W_{j+1}(x_1^*)+\int_{x_1^*}^\infty U_{j+1}(x)dF_{k+1}(x)\\
&\geq \sP(X_{j+1}\leq x_1^*)W_{j+1}(x_1^*)+\sP(X_{j+1}>x_1^*)U_{j+1}(a_{j+1})\\
&=\sP(X_{j+1}'=b_{j+1})W_{j+1}^{(j+1)}(x_1^*)+\sP(X_{j+1}'=a_{j+1})U_{j+1}^{(j+1)}(a_{j+1})\\
&=W_j^{(j+1)}(x_1^*),
\end{align*}
where we used \eqref{eq:drop-the-superscript} in the third equality, and \eqref{eq:stop-if-top-value} in the last equality.
This completes the backward induction. Finally, setting $j=1$ in \eqref{eq:comparison-array} and using \eqref{eq:U-not-changing} and
\begin{equation*}
V_n^{(k)}=\max\{U_1^{(k)}(x_1^*),W_1^{(k)}(x_1^*)\}, \qquad k=1,2,\dots,n,
\end{equation*}
it follows that $V_n^{(k)}\geq V_n^{(k+1)}$ for $k=1,\dots,n-1$. This proves the lemma.
\end{proof}

\begin{proof}[Proof of Theorem \ref{thm:main}] 
The inequality \eqref{eq:main-lower-bound} is an immediate consequence of Lemma \ref{lem:reduction} and Lemma \ref{lem:V-sequence-bound}.
The bound is attained by replacing the V-sequence that attains the bound in Lemma \ref{lem:V-sequence-bound} with a continuous sequence, as follows. 
Let $X_1,\dots,X_n$ be independent random variables such that $X_1$ has the uniform distribution on $(0,1)$,
and for $i=2,\dots,n$, $X_i$ has density function
\begin{equation*}
f_i(x):=\begin{cases}
\frac{1}{n}, & \mbox{if } i<x<i+1, \\
\frac{n-1}{n}, & \mbox{if } -i<x<-i+1, \\
0, & \mbox{otherwise}.
\end{cases}
\end{equation*}
Then, since the supports of $X_1,\dots,X_n$ do not overlap, the optimal win probability $V_n(X_1,\dots,X_n)$ is the same as the optimal win probability for the V-sequence that attains the bound in Lemma \ref{lem:V-sequence-bound}. That is, $V_n(X_1,\dots,X_n)={(1-\frac{1}{n})}^{n-1}$.
\end{proof}

\section*{Acknowledgment}
The authors thank the anonymous referee for several helpful suggestions that led to an improved presentation of the paper.

\end{document}